\documentclass[11pt]{article}
\usepackage{latexsym,amsfonts,amssymb,amsmath,euscript,amsthm}
\usepackage[applemac]{inputenc}  
\usepackage[english,francais]{babel}  
\usepackage{graphicx} 
\usepackage{subfigure}
\usepackage{hyperref}         
\usepackage{cite}
\usepackage{color,fancybox}
\usepackage{float}

\newcommand{\R} {\mathbb{R}}
\newcommand{\N} {\mathbb{N}}

\newcommand{\eps}{\epsilon}
\newcommand{\eto}{\stackrel{\eps\to 0}{\longrightarrow}}

\newcommand{\weto}{\stackrel{\eps\to 0}{\rightharpoonup}}

\def\eto{\buildrel \epsilon\to 0\over\longrightarrow }

\def\kti{\buildrel k\to\infty\over\longrightarrow }

\title{Elliptic and parabolic problems in thin domains with doubly weak oscillatory boundary\footnote{
This first author (JMA) has been partially  supported by grants MTM2016-75465-P,  PID2019-103860GB-I00, ICMAT Severo Ochoa project SEV-2015-0554, MICINN, Spain and Grupo de Investigaci\'on CADEDIF, UCM.
 The second author (MVP) has  been partially supported by grant MTM2016-75465-P MICINN, Spain and Grupo de Investigaci\'on CADEDIF, UCM.}   }
\author{Jos\'{e} M. Arrieta\footnote{Departamento de An\'alisis Matem\'atico y Matem\'atica Aplicada, Universidad Complutense de Madrid, 28040 Madrid and Instituto de Ciencias Matem\'aticas
CSIC-UAM-UC3M-UCM, C/Nicol\'as Cabrera 13-15, Cantoblanco, 28049 Madrid, Spain. Email: arrieta@mat.ucm.es}
and Manuel Villanueva-Pesqueira\footnote{Departamento de Matem\'atica Aplicada, ICAI, Universidad Pontificia de Comillas, Madrid, Spain,  $\qquad\qquad\qquad$ Email: mvillanueva@icai.comillas.edu} 
}

\date{ }

\begin{document}

\maketitle
%
%
%
%
\rightline {Dedicated to Tomás Caraballo on the occasion of his 60-th birthday}

\par\bigskip\bigskip

{\footnotesize 
\par\noindent {\bf Abstract:} In this work we consider higher dimensional thin domains with the property that both boundaries, bottom and top, present oscillations of weak type. We consider the Laplace operator with Neumann boundary conditions and analyze the behavior of the solutions as the thin domains shrinks to a fixed domain $\omega\subset \R^n$. We obtain the convergence of the resolvent of the elliptic operators in the sense of compact convergence of operators, which in particular implies the convergence of the spectra.  This convergence of the resolvent operators will allow us to conclude  the global dynamics, in terms of the global attractors of a reaction diffusion equation in the thin domains. In particular, we show the upper semicontinuity of the attractors and stationary states. An important case treated is the case of a quasiperiodic situation, where the bottom and top oscillations are periodic but with period rationally independent.

\vskip 0.5\baselineskip

%

\vspace{11pt}


\vspace{6pt}

}


%
%

\numberwithin{equation}{section}
\newtheorem{theorem}{Theorem}[section]
\newtheorem{lemma}[theorem]{Lemma}
\newtheorem{corollary}[theorem]{Corollary}
\newtheorem{proposition}[theorem]{Proposition}
\newtheorem{definition}[theorem]{Definition}
\newtheorem{remark}[theorem]{Remark}
\allowdisplaybreaks

\section{Introduction}
\selectlanguage{english}
We are interested in analyzing the behavior of the solutions of certain Partial Differential Equations of elliptic and parabolic type which are posed in a varying domain $R^\epsilon$. This domain is a thin domain in $\R^{n+1}$ which presents an oscillatory behavior at the boundary and it is given as the region between two oscillatory functions, that is, 
\begin{equation}\label{thin-appen}
R^\epsilon = \Big\{ (x, y) \in \R^{n+1} \; | \;  x \in \omega\subset \R^n,  \; -\eps k^1(x,\eps)< y < \epsilon k^2(x,\eps) \Big\},
 \end{equation}
 where $\omega$ is a smooth bounded domain in $\R^{n}$ and the functions $k^1,k^2$ satisfy certain hypotheses, see {\bf (H)} below.  Observe that the lower part of the boundary is described  by the function $k^1$ and the upper part of the boundary by $k^2$. These two function may present oscillations.  See Figure \ref{thin1 appen} for a two-dimensional example.

One equation we are interested in is the following linear elliptic equation
\begin{equation} \label{1OPI0}
\left\{
\begin{gathered}
- \Delta w^\epsilon + w^\epsilon = f^\epsilon
\quad \textrm{ in } R^\epsilon, \\
\frac{\partial w^\epsilon}{\partial \nu^\epsilon} = 0
\quad \textrm{ on } \partial R^\epsilon,
\end{gathered}
\right. 
\end{equation}
where $f^\epsilon \in L^2(R^\epsilon)$, 
and $\nu^\epsilon$ is the unit outward normal to $\partial R^\epsilon$.  But eventually we will be interested in saying something about the behavior of the global dynamics of a reaction diffusion equation of the type 
 
\begin{equation} \label{1OPI0 parabolic}
\left\{
\begin{gathered}
w^\epsilon_t - \Delta w^\epsilon + w^\epsilon = f(w^\epsilon)
\quad \textrm{ in } R^\epsilon, t>0, \\
\frac{\partial w^\epsilon}{\partial \nu^\epsilon} = 0
\quad \textrm{ on } \partial R^\epsilon.
\end{gathered}
\right. 
\end{equation}
 
As we will see and as it is already well known, see for instance \cite{ACar,ACarLo1,ArrCarPerSil,Per},  the analysis of the convergence of \eqref{1OPI0} will basically dictate the behavior of the dynamics of \eqref{1OPI0 parabolic} in terms of the behavior of the global dynamics (continuity of solutions, continuity of equilibria upper and lower semicontinuity of attractors, etc).  As a matter of fact, if the solutions of  \eqref{1OPI0} approach in certain sense the solutions of the limiting problem $-Lw+w=f$ in $\omega$, where $L$ will be an elliptic operator then, the solutions of \eqref{1OPI0 parabolic} will converge in certain sense to the solutions of the equation $w_t - L w^\epsilon + w^\epsilon = f(w^\epsilon)$ in $\omega$.

%

\par\bigskip\bigskip 
 In order to simplify the notation, we write every point in $\R^{n+1}$ as follows
 $$(x,y)\in \R^{n+1}, \hbox{ with } x=(x_1, x_2, \cdots, x_n) \in \R^n \hbox{ and } y \in \R.$$

Let $k^i$ be a function, $i=1,2$,
\begin{equation*}
\begin{array}{rl}
k^i: \omega\times (0,1) &\longrightarrow \R^+ \\
 (x,\eps)&\longrightarrow k^i(x,\eps)=k_\eps^i(x),
 \end{array}
 \end{equation*}
 such that
 \begin{itemize}
\item[\bf{(H.1)}] $k_\eps^i$ is a $C^1$ function in the first variable and
\begin{equation}\label{main}
\eps\Big|\frac{\partial k^i_\eps}{\partial x_j}(x)\Big|\eto 0 \hbox{ uniformly in } \omega, \quad j=1,\dots,n.
\end{equation}
\item[\bf{(H.2)}] There exist two positive constants independent of $\eps$ such that 
\begin{equation}\label{boundk}
0<C^i_1\leq k_\eps^i( \cdot) \leq C^i_2.
\end{equation}
\item[\bf{(H.3)}] There exists a function $K^i$ in $L^2(\omega)$ such that
$$k^i_\eps \weto K^i \hbox{ w-}L^2(\omega).$$
\end{itemize}


 Indeed, the thickness of the domain has order $\eps$ and we say that the domain presents weak oscillations due to the convergence \eqref{main}. An interesting particular example of the introduced general setting is the case where  
 \begin{equation}\label{per}
 k^1_\eps(x)=h(x/\eps^\alpha), \quad k^2_\eps(x)=g(x/\eps^\beta),
 \end{equation}
with $0<\alpha, \beta<1$ and the functions $g,h \,: \R^n \to \R $ are $C^1$ periodic functions.

  \begin{figure}[H]
  \centering
    \includegraphics [width=8cm, height=3.5cm]{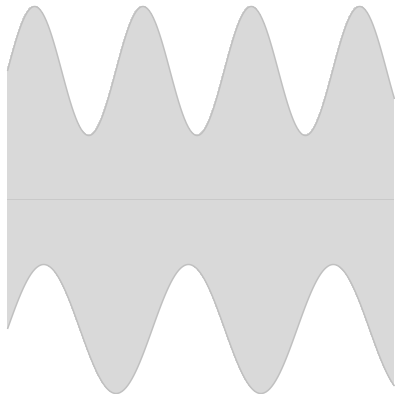}
    \caption{Thin domain $R^\eps$ with doubly weak oscillatory boundary}
    \label{thin1 appen}
\end{figure}

The variational formulation of (\ref{1OPI0}) is the following:  find $w^\epsilon \in H^1(R^\epsilon)$ such that 
\begin{equation} \label{VFP}
\int_{R^\epsilon} \Big\{ \nabla w^\eps \nabla \varphi
+ w^\epsilon \varphi \Big\} dx_1 dx_2 = \int_{R^\epsilon} f^\epsilon \varphi dx_1 dx_2, 
\, \forall \varphi \in H^1(R^\epsilon).
\end{equation}

Observe that, for fixed $\eps>0$, the existence and uniqueness of solution to problem (\ref{1OPI0}) is guaranteed by Lax-Milgram Theorem. Then, we will analyze the behavior of the solutions as the parameter $\eps$ tends to zero. In particular, we first perform suitable change of variables which allows us to substitute, in some sense, the original problem \eqref{1OPI0} posed in an $(n+1)$-oscillating thin domain into a simpler problem with oscillating coefficients posed in an $n-$dimensional fixed domain. Notice that, this fact is in agreement with the intuitive idea that the family of solutions $u^\eps$ should converge to a function of just $n$ variables as $\eps$ goes to zero since the domain is thin. Subsequently,  by using the previous results and adapting well-known techniques in homogenization we obtain explicitly the homogenized limit problem for the interesting cases where $k^i_\eps$, $i=1,2$, satisfy \eqref{per}. In this respect, notice that the coupled effect of both oscillating boundaries requiere to consider more general settings than classical periodicity, for instance, quasi-periodicity or reiterated homogenization framework.

It is worth observing that we cannot apply a direct homogenization technique to obtain the limit problem of \eqref{1OPI0} since there is not a representative cell which describe the domain and, moreover, the effect of the oscillations at both boundaries, top and bottom, cannot be separated in an easy way since at the same time that the boundaries oscillate, the domain is shrinking and therefore the effect of the oscillations at one boundary is coupled in a nontrivial way with the oscillations at the other boundary.

Finally, we show that the performed analysis of the elliptic linear problem \eqref{1OPI0} together with known results from the theory on nonlinear dynamics of dissipative systems allows to analyze the convergence properties of the solutions and attractors of the following semilinear parabolic evolution equation

\begin{equation} \label{1OPI0 parabolic}
\left\{
\begin{gathered}
w^\epsilon_t - \Delta w^\epsilon + w^\epsilon = f(w^\epsilon)
\quad \textrm{ in } R^\epsilon, t>0, \\
\frac{\partial w^\epsilon}{\partial \nu^\epsilon} = 0
\quad \textrm{ on } \partial R^\epsilon,
\end{gathered}
\right. 
\end{equation}
where $\nu^\epsilon$ is the unit outward normal to $\partial R^\epsilon$ and the function $f \,: \R \to \R $ is a $C^2-$ function with bounded derivatives. Moreover, since we are interested in the behavior of solutions as $t\to \infty$ and its dependence with respect to the small parameter $\eps$, we will require that the solutions of \eqref{1OPI0 parabolic} be bounded for large values of time. A natural assumption to obtain this boundedness of the solutions is expressed in the following dissipative condition
\begin{equation}\label{dissipative}
\limsup_{|s|\to \infty} \frac{f(s)}{s}<0.
\end{equation}

In order to accomplish our goal, we consider here the linear parabolic problems associated to the perturbed equation \eqref{1OPI0 parabolic} and its limit in the abstract framework given by \cite{Ha,He81}. Then, we show that under an appropriate notion of convergence, the spectral convergence and the convergence of the corresponding linear semigroups may be obtained. Moreover, we are in conditions to transfer this information to the nonlinear dynamics through the variation of constants formula. This way, the continuity of nonlinear semigroups and the upper semicontinuity of the family of attractors, as well as, the upper semicontinuity of the set of stationary states can be established. We refer to the fundamental works \cite{Ha}, \cite{BabVis},\cite{Tem},\cite{SelYou},\cite{CarLanRob} to understand the asymptotic dynamics of evolutionary equations and the behavior under perturbations. Se also, \cite{KapVal, KapKasVal} and references therein for related works.

We would like to mention that there are several papers addressing the problem of studying the effect of rough boundaries on the behavior of the solution of partial differential equations.  Among others, we can mention \cite{C1, komo} in the context of flows, \cite{BGG07A, AnBra} from the point of view of the elasticity and \cite{BePyChe} where complete asymptotic expansions of the solutions were studied. In particular, let us point out that the convergence properties of the solutions of the elliptic equation \eqref{1OPI0} have been discussed in several papers in the literature for different kind of thin domains. It is well known which is the limit if the thin domain does not present oscillations,
see for instance, \cite{HR, R, SP, ArrSan}.  The case where the thin domain presents weak roughness, $k^1_\eps(x)=h(x/\eps^\alpha), \quad k^2_\eps(x)=0$ with  $0<\alpha <1$, was treated in \cite{Arr} using changes of variables and rescaling the thin domain as in classical works in thin domains with no oscillations, see for instance \cite{HR, R}. The resonant case, $k^1_\eps(x)=h(x/\eps^\alpha), \quad k^2_\eps(x)=0$ with  $\alpha= 1$, was studied in \cite{ArrCarPerSil, MP2} using standard techniques in homogenization,  see \cite{BenLioPap, CioPau, SP, CiorDonato} for a general introduction to the homogenization theory. More recently, the homogenized limit problem for the case of thin domains with a fast oscillatory boundary, $k^1_\eps(x)=h(x/\eps^\alpha), \quad k^2_\eps(x)=0$ with $\alpha>1$, was obtained in \cite{ArrPer2013} by decomposing the domain in two parts separating the oscillatory boundary. Moreover, in \cite{ArrVi2017} the previous cases were analyzed in a unified way adapting the recent unfolding operator method.

Indeed, we think that considering thin domains with doubly oscillatory boundary is a natural way to extend the models studied in previous papers to a more realistic situations where several microscopic scales appear. In fact, understanding how the complicated micro geometry of the thin structures affects the macro properties of the material is a very relevant issue in engineering and applied science. Moreover, let us point out since upper and lower boundary present different orders of frequency and profiles of oscillation it is not possible to analyze this kind of problems using a simple generalization of the methods applied to thin domains with only one oscillatory boundary.

 The case where one of the two oscillatory boundary presents an extremely high oscillatory behavior, $k^1_\eps(x)=h(x/\eps^\alpha), \quad k^2_\eps(x)=g(x/\eps^\beta)$ with $\alpha >1$, is well-know in the literature, see \cite{ArrPer2013} for the periodic case and \cite{Per} for the locally periodic case where the parabolic problem is also studied. The authors of both papers combine the classical oscillatory test functions method of Tartar together with an adaptation of the method from \cite{ArrPer2013} in order to get the homogenized limit problem. However the approach introduced in both papers is essentially based on the fact that $\alpha>1$.  
 
Let us point out that despite the works mentioned above, the case introduced in this paper has not been treated previously in the literature. As a main novelty, we approach a problem posed in a general thin domain in $\R^{n+1}$ with weak oscillations at top and bottom boundary, see assumption \eqref{main} without extra conditions on the periodicity by a more studied equation with oscillating coefficients in $\R^n$. The crucial idea is that since the oscillations at the boundary are not too ``wild'' we can transform the original problem in $R^\eps$ into a one less dimensional simpler problem  posed in a fixed domain through diffeomorphisms which will depend on the parameter $\eps$.

This paper is organized as follows.  In Section 2 we set up the notation and state important results which allows to reduce the problem to study the behavior of the solutions of an equation with oscillating coefficients in a fixed domain. In Section 3 and Section 4 we study the convergence properties of the elliptic and parabolic equation combining results of homogenization theory and the theory on nonlinear dynamics of dissipative systems 

We have given detailed proofs of the new results, when the weak oscillating boundaries play an important role, while, for the proofs involving routine procedures of homogenization theory or nonlinear dynamics of dissipative systems theory we refer to well-known results in the literature.

 \section{Reduction to an $n$-dimensional problem with oscillating coefficients}

In this section we concentrate in the study of the elliptic problem \eqref{1OPI0}
and start analyzing the behavior of the solutions as $\eps\to 0$. As a matter of fact, we will be able to reduce the study of \eqref{1OPI0} in the thin domain $R^\eps$ to the study of an elliptic problem with oscillating coefficients in the lower dimensional fixed domain $\omega$.   This dimension reduction will be the key point to obtain the correct limiting equation. 
In order to state the main result of this section, let us first make some definitions. We will denote by 

\begin{equation} \label{def-eta}
\displaystyle \eta^i(\eps)=\max_{i \in \{1\cdots n\}}\Big\{\sup_{x \in \omega}\Big|\eps \frac{\partial k^i_\eps}{\partial x_i}(x) \Big|\Big\}>0,\quad i=1,2\quad \hbox{ and }\quad  \eta(\eps)=\eta^1(\eps)+\eta^2(\eps).
\end{equation}
Observe that from hypothesis {\bf (H.1)} we have $\eta(\eps)\eto 0$

Also, we denote by $K_\eps(x)=k^1_\eps(x)+k_\eps^2(x)$ (that is $\eps K_\eps(x)$ is the thickness of the thin domain $R^\eps$ at the point $x\in \omega$),   $\displaystyle \hat f^\eps(x)=\frac{1}{\eps K_\eps(x)}\int_{-\eps k^1_\eps(x)}^{\eps k^2_\eps(x)} f^\eps(x,y)dy$  and we consider the problem:

\par\bigskip

 \begin{equation}\label{transformed-problem-hat}
\left\{
\begin{gathered}
- \frac{1}{K_\eps}\sum_{i=1}^n\frac{\partial}{\partial x_i}\Big(K_\eps \frac{\partial \hat w^\eps}{\partial x_i}\Big)+ \hat w^\eps = \hat f^\eps
\quad \textrm{ in } \omega, \\
\frac{\partial \hat w^\eps}{\partial \eta}  = 0  \quad \textrm{ on } \partial \omega.
\end{gathered}
\right. 
\end{equation}

Our main result in this section is

\begin{proposition}\label{main-reduction}
There exists a constant $C$ independent of $\eps>0$ such that for all $f^\eps\in L^2(R^\eps)$, we have 
\begin{equation}\label{main-estimate}
\|w^\eps-\hat w^\eps\|_{H^1(R_\eps)}^2\leq C\eta(\eps) \|f^\eps\|^2_{L^2(R_\eps)}.
\end{equation}
\end{proposition} 

\par\bigskip

In order to prove this result, we will need to obtain first some preliminary lemmas.  
We start transforming equation \eqref{1OPI0} into an equation in the modified thin domain
\begin{equation}\label{thin-intro1 appen}
R_a^\epsilon = \Big\{ (x,\bar y) \in \R^{n+1} \; | \;  x \in \omega,  \; 0 < \bar y < \epsilon \, k_\eps^2(x) + \eps \,k_\eps^1(x) \Big\}.
 \end{equation}
(see Figure \ref{thin2 appen}) where it can be seen that we have transformed the oscillations of both boundaries into oscillations of just the boundary at the top.  For this, we considering  the following family  of diffeomorphisms
\begin{equation*}
\begin{array}{rl}
L^\eps: R_a^\epsilon &\longrightarrow R^\eps \\
 (\bar x, \bar y)&\longrightarrow (x, y) := (\bar x, \bar y - \eps \,k^1_\eps(\bar x) ).
 \end{array}
\end{equation*}
Notice that the inverse of this diffeomorphism is  $(L^\eps)^{-1}(x,y)=(x, y+\eps k^1_\eps(x))$. Moreover, from the structure of these diffeomorphisms and hypothesis {\bf (H1)} we easily get that there exists a constant $C$ such that the Jacobian Matrix of $L_\eps$ and $L_\eps^{-1}$ satisfy 
\begin{equation}\label{bound-jacobian}
\|JL^\eps\|_{L^\infty},\|J(L^\eps)^{-1}\|_{L^\infty}\leq C.
\end{equation}
Moreover, we also have  $det(JL^\eps)(\bar x, \bar y)=det(J(L^\eps)^{-1})( x, y)=1$.

We will show that the study of the limit behavior of the solutions of \eqref{1OPI0} is equivalent to analyze the behavior of the solutions of the following problem
\begin{equation} \label{OPI01 appen}
\left\{
\begin{gathered}
- \Delta v^\epsilon + v^\epsilon = f_1^\eps
\quad \textrm{ in } R_a^\epsilon, \\
\frac{\partial v^\epsilon}{\partial \nu^\epsilon} = 0
\quad \textrm{ on } \partial R_a^\epsilon,
\end{gathered}
\right. 
\end{equation}
where $\nu^\epsilon$ is the unit outward normal to $\partial R_a^\epsilon$ and 
\begin{equation}\label{def-f1eps}
f_1^\eps= f^\eps \circ L^\eps.
\end{equation}

%
 \begin{figure}[H]
  \centering
    \includegraphics [width=8cm, height=3cm]{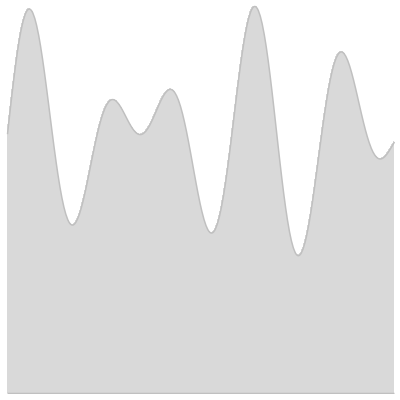}
    \caption{Thin domain $R_a^\epsilon$ obtained from $R^\eps$ of Figure \ref{thin1 appen} }
    \label{thin2 appen}
\end{figure}

Notice that 
$$\|f_1^\eps\|_{L^2(R^\eps_a)}^2=\int_{R^\eps_a}|f_\eps\circ L_\eps(\bar X)|^2d\bar X=\int_{R^\eps}|f_\eps(X)|^2|det(J(L_\eps^{-1}))(X)|dX=\|f^\eps\|_{L^2(R^\eps)}^2,$$
for all $f_\eps\in L^2(R^\eps)$. In particular, this implies that
\begin{equation}\label{a priori1 appen}
\|v_\eps\|_{H^1_a(R^\eps)}\leq \|f_1^\eps\|_{L^2(R^\eps_a)}=\|f^\eps\|.
_{L^2(R^\eps_a)}.
\end{equation}

\par\medskip

\begin{lemma}\label{1 transform}
Let $w^\eps$ and $v^\eps$ be the solutions of problems \eqref{1OPI0} and \eqref{OPI01 appen} respectively. Then, there is a constant $C>0$, independent of $\eps$, such that 
\begin{equation}\label{eq-Lemma}
||(w^\eps\circ L^\eps) - v^\eps||_{H^1(R_a^\epsilon)}^2 \leq C\eta^1(\eps)\|f^\eps\|_{L^2(R_\eps)}^2.
\end{equation}
\end{lemma}

\begin{proof}
From the definition of $L^\eps$ we have
\begin{align*}
& \frac{\partial (w^\epsilon\circ L^\eps)}{\partial x_i} =  \frac{\partial w^\epsilon}{\partial x_i}  - \eps\Big(\frac{\partial k^1_\eps}{\partial x_i}(x)\Big) \frac{\partial w^\epsilon}{\partial y}, \quad i=1,\dots,n,\\
& \frac{\partial (w^\epsilon\circ L^\eps)}{\partial \bar y} =  \frac{\partial w^\epsilon}{\partial y}.
 \end{align*}
In the new system of variables ($x=x$ and $\bar y = y + \eps \,k^1_\eps(x)$ ) the variational formulation of \eqref{1OPI0} is given by
\begin{equation}\label{VFP0 appen}
\begin{split}
&\int_{R_a^\epsilon} \Big\{ \sum_{i=1}^n\frac{\partial (w^\epsilon\circ L^\eps)}{\partial x_i} \frac{\partial \varphi}{\partial x_i} + \frac{\partial(w^\epsilon\circ L^\eps)}{\partial \bar y}\frac{\partial \varphi}{\partial \bar y} + (w^\epsilon\circ L^\eps)\varphi \Big\} dx d\bar y\\
&\quad +\int_{R_a^\epsilon} \sum_{i=1}^n \Big\{\eps \frac{\partial k^1_\eps}{\partial x_i}(x)\Big(\frac{\partial(w^\epsilon \circ L^\eps)}{\partial \bar y} \frac{\partial \varphi}{\partial x_i} + \frac{\partial (w^\epsilon\circ L^\eps)}{\partial x_i}\frac{\partial \varphi}{\partial \bar y}\Big)\Big\}dx d \bar y\\
& \quad+ \int_{R_a^\epsilon} \Big\{\sum_{i=1}^n \eps \Big(\frac{\partial k^1_\eps}{\partial x_i}(x)\Big)^2\frac{\partial(w^\epsilon \circ L^\eps)}{\partial \bar y} \frac{\partial \varphi}{\partial \bar y} \Big\} dx d \bar y\\
& = \int_{R_a^\epsilon} f^\eps_1 \varphi\, dx d \bar y, 
\quad \forall \varphi \in H^1(R_a^\epsilon).
\end{split}
\end{equation}
On the other hand, the weak formulation of \eqref{OPI01 appen} is: find $v^\epsilon \in H^1(R_a^\epsilon)$ such that 
\begin{equation} \label{VFP1 appen}
\int_{R_a^\epsilon} \Big\{ \sum_{i=1}^n \frac{\partial v^\epsilon}{\partial x_i} \frac{\partial \varphi}{\partial x_i} 
+ \frac{\partial v^\epsilon}{\partial \bar y} \frac{\partial \varphi}{\partial \bar y}
+ v^\epsilon \varphi \Big\} \,dx d\bar y = \int_{R_a^\epsilon} f^\eps_1 \varphi \, \,dx d\bar y, 
\quad \forall \varphi \in H^1(R_a^\epsilon).
\end{equation}
Therefore, subtracting \eqref{VFP1 appen} from \eqref{VFP0 appen}, taking $(v^\eps-w^\epsilon\circ L^\eps)$ as a test function and after some computations and simplifications, we obtain

\begin{align*}
\|w^\eps\circ L^\eps-v_\eps\|_{H^1(R_a^\eps)}^2
\leq \eta^1_\eps\|\nabla (w^\eps\circ L^\eps)\|_{L^2(R_a^\eps)}\|\nabla (w^\eps\circ L^\eps-v_\eps)\|_{L^2(R_a^\eps)}.
\end{align*}

This implies 
\begin{align*}
\|w^\eps\circ L^\eps-v_\eps\|_{H^1(R_a^\eps)}
\leq \eta^1_\eps\|w^\eps\circ L^\eps\|_{H^1(R_a^\eps)}\leq \eta^1_\eps\|w^\eps\circ L^\eps-v_\eps\|_{H^1(R_a^\eps)}+ \eta^1_\eps\|v_\eps\|_{H^1(R_a^\eps)},
\end{align*}
and therefore
\begin{align*}
\|w^\eps\circ L^\eps-v_\eps\|_{H^1(R_a^\eps)}
\leq \frac{\eta^1_\eps}{1-\eta^1_\eps}\|v_\eps\|_{H^1(R_a^\eps)}\leq \frac{\eta^1_\eps}{1-\eta^1_\eps}\|f_\eps\|_{L^2(R^\eps_a)},
\end{align*}
where we use that $\|v_\eps\|_{H^1(R_a^\eps)}\leq \|f_1^\eps\|_{L^2(R_a^\eps)}=\|f^\eps\|_{L^2(R^\eps)}$ since $v_\eps$ is the solution of \eqref{OPI01 appen}. This  proves the result. 
\end{proof}

\begin{corollary}
With the same hypothesis of the previous Lemma, we also have that there exists a constant $C>0$ (probably different from the one in the previous Lemma) such that 
$$||w^\eps - v^\eps\circ (L^\eps)^{-1}||_{H^1(R^\epsilon)}^2 \leq C\eta^1(\eps)\|f^\eps\|_{L^2(R_\eps)}^2.$$
\end{corollary}
\begin{proof} Just perform the change of variables with the diffeomorphism $(L^\eps)^{-1}$ to \eqref{eq-Lemma} and use \eqref{bound-jacobian}. \end{proof}

\begin{remark}
Notice that from the point of view of the limit behavior of the solutions it is the same to study problem \eqref{1OPI0} defined in the doubly oscillating thin domain as to analyze problem \eqref{OPI01 appen} posed in a thin domain with just one oscillating boundary. It is important to note that this is true because of {\bf(H.1)}. If that assumption is not satisfied, at least for $k_\eps^1$, then the simplification it is not possible. For instance, if we have a domain with oscillations with the same period in both boundaries like this one
$$R^\epsilon = \Big\{ (x, y) \in \R^2 \; | \;  x \in (0,1),  \; -\eps(2 - g(x/\eps)) < y < \epsilon g(x/\eps)  \Big\},$$
where  $g: \R \to \R $ is a smooth  $L-$periodic function.
Observe that in this particular case we have
$$k^1_\eps(x)=2 - g(x/\eps), \quad k^2_\eps(x)=g(x/\eps).$$
Therefore, it follows straightforward that condition \eqref{main} is not satisfied.
Observe that the original problem \eqref{1OPI0} for this particular thin domain is in the framework of the classical periodic homogenization while the converted problem \eqref{OPI01 appen} is posed in a rectangle of height $\eps$ where homogenization theory is not necessary to analyze the behavior of the solutions. Indeed, if $k_\eps^1$ does not satisfy \eqref{main} the solutions of problems \eqref{1OPI0} and \eqref{OPI01 appen} are not comparable in general.
 \end{remark}
 
 Now we define a transformation on the thin domain $R^\eps_a$, which will map $R^\eps_a$ into the fixed cylindrical domain
$Q=\omega \times (0,1)$. This transformation is given by
\begin{equation*}
\begin{array}{rl}
S^\eps: Q &\longrightarrow R_a^\epsilon \\
 (x, y)&\longrightarrow (\bar x,  \bar y) := (x, y\,\eps K_\eps(x)).
 \end{array}
 \end{equation*}
We recall that $K_\eps(x)=  k_\eps^2(x) + k_\eps^1(x)$.  

Using the chain rule and standard computations it is not difficult to see that  there exists $c,C>0$  such that if $V^\eps\in H^1(R_a^\eps)$ and $U^\eps=V^\eps \circ S^\eps\in H^1(Q)$ then  the following estimates hold

\begin{equation}\label{comparison-L2}
c\eps^{-1}\|V^\eps\|^2_{L^2(R^\eps_a)}\leq \|U^\eps\|^2_{L^2(Q)}\leq C\eps^{-1}\|V^\eps\|^2_{L^2(R^\eps_a)},
\end{equation}

\begin{equation}\label{comparison-Dy}
c\eps^{-1}\|\frac{\partial V^\eps}{\partial \bar y}\|^2_{L^2(R^\eps_a)}\leq \frac{1}{\eps^2}\|\frac{\partial U^\eps}{\partial  y}\|^2_{L^2(Q)}\leq C\eps^{-1}\|\frac{\partial V^\eps}{\partial \bar y}\|^2_{L^2(R^\eps_a)},
\end{equation}

\begin{equation}\label{comparison-H1}
c\eps^{-1}\|\nabla V^\eps\|^2_{L^2(R^\eps_a)}\leq \|\nabla U^\eps\|^2_{L^2(Q)}\leq C\eps^{-1}\|\nabla V^\eps\|^2_{L^2(R^\eps_a)}.
\end{equation}

\par\medskip

Now, under this change of variables and defining $u^\eps=v^\eps\circ S^\eps$ where $v^\eps$ satisfies \eqref{OPI01 appen} and $f^\eps_2= f^\eps_1 \circ S^\eps$, where $f_1^\eps$ is defined in \eqref{def-f1eps}, problem \eqref{OPI01 appen} becomes
\begin{equation}\label{transformed-problem}
\left\{
\begin{gathered}
- \frac{1}{K_\eps}\hbox{div}\big(B^\eps(u^\eps)\big)  + u^\epsilon = f^\eps_2
\quad \textrm{ in } Q, \\
B(u^\eps)\cdot \eta  = 0
\quad \textrm{ on } \partial Q,\\
u^\eps=v^\eps \circ S^\eps \textrm{ in } Q,
\end{gathered}
\right. 
\end{equation}
where $\eta$ denotes the unit outward normal vector field to $\partial Q$ and every coordinate of $B(u^\eps)$ is defined as follows
\begin{align*}
& B(u^\eps)_i = K_\eps \frac{ \partial u^\eps}{\partial x_i} - y \frac{\partial K_\eps}{\partial x_i}\frac{ \partial u^\eps}{\partial y}, \quad i=1,\dots, n,\\
& B(u^\eps)_{n}= \sum_{i=1}^n\Big( - y \frac{\partial K_\eps}{\partial x_i}\frac{ \partial u^\eps}{\partial x_i} + \frac{y^2}{ K_\eps}\Big(\frac{\partial K_\eps}{\partial x_i}\Big)^2\frac{ \partial u^\eps}{\partial y}\Big) +\frac{1}{\eps^2 K_\eps}\frac{ \partial u^\eps}{\partial y}.
\end{align*}

Notice that in the new system of coordinates we obtain a domain which is neither thin nor oscillating anymore. In some sense, we have substituted the oscillating thin domain by
oscillating coefficients in the differential operator.

In order to analyze the limit behavior of the solutions of  \eqref{transformed-problem} we establish the relation to the solutions of the following easier problem
\begin{equation}\label{transformed-problem1}
\left\{
\begin{gathered}
- \frac{1}{K_\eps}\sum_{i=1}^n\frac{\partial}{\partial x_i}\Big(K_\eps \frac{\partial w_1^\eps}{\partial x_i}\Big) + \frac{1}{\eps^2 K_\eps} \frac{\partial^2w_1^\eps}{\partial y^2}  + w_1^\epsilon = f^\eps_2
\quad \textrm{ in } Q, \\
\frac{\partial w_1^\eps}{\partial \eta}  = 0  \quad \textrm{ on } \partial Q.
\end{gathered}
\right. 
\end{equation}

Observe that under the assumptions on the functions $k^1_\eps$ and $k^2_\eps$, equation \eqref{transformed-problem1} admits a unique solution $w_1^\eps \in H^1(Q)$, which satisfies the a priori
estimates
\begin{equation} \label{EST1 appen}
\begin{gathered}
\| w_1^\epsilon \|_{L^2(Q)}, \, \, \, \Big\| \frac{\partial w_1^\epsilon}{\partial x_i} \Big\|_{L^2(Q)}, \, \, \, \frac{1}{\epsilon} \Big\| \frac{\partial w_1^\epsilon}{\partial y} \Big\|_{L^2(Q)} 
\le C\|f_2^\eps\|_{L^2(Q)}, \quad i=1\cdots n.
\end{gathered}
\end{equation}

\begin{lemma}\label{2 transform}
Let $u^\eps$ and $w_1^\eps$ be the solution of problems \eqref{transformed-problem} and \eqref{transformed-problem1} respectively. Then, we have
$$\sum_{i=1}^n\Big|\Big|\frac{\partial(u^\eps - w_1^\eps)}{\partial x_i}\Big|\Big|_{L^2(Q)}^2 + \frac{1}{\eps^2}\Big|\Big|\frac{\partial(u^\eps - w_1^\eps)}{\partial y}\Big|\Big|_{L^2(Q)}^2 +||u^\eps - w_1^\eps||_{L^2(Q)}^2 \leq C\eta(\eps) \eps^{-1}\|f_\eps\|_{L^2(R_\eps)}^2, $$
where $\eta(\eps)$ is defined in \eqref{def-eta}.
\end{lemma}

\begin{proof}
Subtracting the weak formulation of \eqref{transformed-problem1} from the weak formulation of \eqref{transformed-problem} and choosing $w_1^\eps-u^\eps $ as test function we get
\begin{equation} \label{21 transform}
\begin{split}
\int_{Q}&\Big\{ \sum_{i=1}^n K_\eps\Big(\frac{\partial (u^\epsilon - w_1^\eps)}{\partial x_i}\Big)^2 + \frac{1}{\eps^2 K_\eps}\Big(\frac{\partial (u^\epsilon - w_1^\eps)}{\partial y}\Big)^2 + 
\big(u^\epsilon - w_1^\eps\big)^2\Big\} dx dy\\
&= \int_{Q} \Big\{ \sum_{i=1}^n y\frac{\partial u^\eps}{\partial y}\frac{\partial K_\eps}{\partial x_i}\frac{\partial (w_1^\eps-u^\eps )}{\partial x_i} + \sum_{i=1}^n y\frac{\partial K_\eps}{\partial x_i}\frac{\partial u^\eps}{\partial x_i}\frac{\partial (w_1^\eps-u^\eps )}{\partial y} - \sum_{i=1}^n \frac{y^2}{K_\eps}\Big(\frac{\partial K_\eps}{\partial x_i}\Big)^2\frac{\partial u^\eps}{\partial y}\frac{\partial (w_1^\eps-u^\eps )}{\partial y}\Big\}dxdy.
\end{split}
\end{equation}
Taking into account that  $\eps |\frac{\partial K_\eps}{\partial x_i}|\leq \eta(\eps)$, convergence \eqref{main},  estimates \eqref{comparison-L2}, \eqref{comparison-Dy}, \eqref{comparison-H1} applied to $u^\eps$ and $v^\eps$ and the a priori estimates of $v^\eps$ (see \eqref{a priori1 appen}), $u^\eps$ and $w_1^\eps$ (see  \eqref{EST1 appen}) and following standard computations, similar to the ones in the proof of Lemma \ref{1 transform}, 
we obtain the result.
\end{proof}

Finally, we compare the behavior of the solution of \eqref{transformed-problem1} to the solutions of the following problem posed in $\omega \subset \R^n$
\begin{equation}\label{transformed-problem3}
\left\{
\begin{gathered}
- \frac{1}{K_\eps}\sum_{i=1}^n\frac{\partial}{\partial x_i}\Big(K_\eps \frac{\partial u_1^\eps}{\partial x_i}\Big)+ u_1^\epsilon = f^\eps_3
\quad \textrm{ in } \omega, \\
\frac{\partial u_1^\eps}{\partial \eta}  = 0  \quad \textrm{ on } \partial \omega,
\end{gathered}
\right. 
\end{equation}
where $f^\eps_3(x)=\displaystyle\int_{0}^{1}f^\eps_2(x,y) dy$ for a.e. $x \in \omega$, which is a function depending only on the $x$ variable. 


Then, considering $u_1^\eps(x)$ as a function defined in $Q$ (extending it in a constant way in the $y$ direction) we prove the following lemma.
\begin{lemma}\label{3 transform}
Let $u^\eps_1$ and $w_1^\eps$ be the solution of problems \eqref{transformed-problem3} and \eqref{transformed-problem1} respectively. Then, we have
$$\sum_{i=1}^N\Big|\Big|\frac{\partial(u^\eps_1 - w_1^\eps)}{\partial x_i}\Big|\Big|_{L^2(Q)}^2 + \frac{1}{\eps^2}\Big|\Big|\frac{\partial w_1^\eps}{\partial y}\Big|\Big|_{L^2(Q)}^2 +||u^\eps_1 - w_1^\eps||^2_{L^2(Q)}\leq C\|f^\eps\|_{L^2(R^\eps)}^2.$$
\end{lemma}
\begin{proof}
Taking $w_1^\eps - u_1^\eps$ as test function in the variational formulation of 
\eqref{transformed-problem1} and \eqref{transformed-problem3}
and subtracting both weak formulations we obtain
\begin{equation} \label{31 transform}
\begin{split}
\int_{Q}&\Big\{ \sum_{i=1}^n K_\eps\Big(\frac{\partial (u^\epsilon_1 - w_1^\eps)}{\partial x_i}\Big)^2 + \frac{1}{\eps^2 K_\eps}\Big(\frac{\partial w_1^\eps}{\partial y}\Big)^2 + 
K_\eps \big(u_1^\epsilon - w_1^\eps\big)^2\Big\} dx dy\\
&= \int_{Q} K_\eps (f_2^\eps - f_3^\eps) u_1^\eps dxdy - \int_{Q}K_\eps (f_2^\eps - f_3^\eps) w_1^\eps dxdy.
\end{split}
\end{equation}
Now we analyze the two terms in the right hand-side. First, taking into account the definition of $f_3^\eps$ we have that for any function $\varphi$ defined in $\omega$, that is, $\varphi=\varphi(x)$ 
$$\int_{Q} K_\eps (f_2^\eps - f_3^\eps) \varphi dxdy=  \int_{\omega} \varphi K_\eps \Big(\int_0^1 f_2^\eps dy - f_3^\eps \Big)dx=0.$$
In particular  $\displaystyle \int_{Q} K_\eps (f_2^\eps - f_3^\eps) u_1^\eps dxdy=0$
and $\displaystyle \int_{Q} K_\eps (f_2^\eps - f_3^\eps) w_1^\eps(x,0)dxdy=0$. Hence, using Holder inequality and \eqref{EST1 appen} we get
\begin{align*}
&\Big|\int_{Q} K_\eps (f_2^\eps - f_3^\eps) w_1^\eps dxdy \Big| = \Big|\int_{Q} K_\eps (f_2^\eps - f_3^\eps) (w_1^\eps (x,y)- w_1^\eps(x,0) ) dxdy\Big|\\
&\qquad \leq \|K_\eps (f_2^\eps - f_3^\eps )\|_{L^2(Q)} \|w_1^\eps - w_1^\eps(x,0) \|_{L^2(Q)}\leq C\|f_2^\eps\|_{L^2(Q)} \|w_1^\eps - w_1^\eps(x,0) \|_{L^2(Q)} \\
&\qquad = C\|f_2^\eps\|_{L^2(Q)} \Big(\int_{Q} \Big| \int_0^y \frac{\partial w_1^\eps}{\partial s}(x,s) ds \Big|^2\Big)^{1/2} \leq C \|f_2^\eps\|_{L^2(Q)}\Big\|\frac{\partial w_1^\eps}{\partial y}  \Big\|_{L^2(Q)}\leq \|f^\eps\|_{L^2(R^\eps)}^2. 
\end{align*}
Therefore, from \eqref{31 transform} the lemma is proved.
\end{proof}

After this lemmas we can provide a proof of the main result of this section.

\par\bigskip 

\par\noindent {\sl Proof of Proposition \ref{main-reduction}.}  Notice first that $u^\eps_1$, the solution of \eqref{transformed-problem3} coincides with $\hat w$, the solution of \eqref{transformed-problem-hat} and the function appearing in the statement of the proposition.

Hence,
$$\|w^\eps-\hat w^\eps\|^2_{H^1(R^\eps)}=\|w^\eps- u^1_\eps\|^2_{H^1(R^\eps)}\leq C\|w^\eps\circ L^\eps - u^1_\eps\circ L^\eps \|^2_{H^1(R^\eps_a)}=C\|w^\eps\circ L^\eps - u^1_\eps \|^2_{H^1(R^\eps_a)}, $$
where we use \eqref{bound-jacobian} and the fact that $u^1_\eps$ does not depend on the $y$ variable. 

But,
$$\|w^\eps\circ L^\eps - u^1_\eps \|^2_{H^1(R^\eps_a)}\leq 2\|w^\eps\circ L^\eps - v^\eps\|^2_{H^1(R^\eps_a)}+2\|v^\eps-u^1_\eps \|^2_{H^1(R^\eps_a)}\leq C\eta(\eps)\|f^\eps\|^2_{L^2(R^\eps)}+2\|v^\eps-u^1_\eps \|^2_{H^1(R^\eps_a)},$$
where we have used the inequalith $(a+b)^2\leq 2a^2+2b^2$ and Lemma \ref{1 transform}. 

Now, using that $u^\eps=v^\eps\circ S^\eps$ and \eqref{comparison-L2}, \eqref{comparison-Dy}, \eqref{comparison-H1} and that $u^1_\eps$ is independent of $y$ and therefore $u^1_\eps\circ S^\eps=u^1_\eps$, we get 
$$\|v^\eps-u^1_\eps \|^2_{H^1(R^\eps_a)}\leq C \big(\eps\|u^\eps -u^1_\eps \|^2_{H^1(Q)}+\frac{1}{\eps}\|\frac{\partial u^\eps}{\partial y}\|^2_{L^2(Q)}\big),$$
and applying now the triangular inequality, 
$$\leq C\big( \eps \|u^\eps-w^1_\eps\|^2_{H^1(Q)}+\frac{1}{\eps}\|\frac{\partial (u^\eps-w^1_\eps)}{\partial y}\|^2_{L^2(Q)}+ \eps \|w^1_\eps-u^1_\eps\|^2_{H^1(Q)}+\frac{1}{\eps}\|\frac{\partial w^1_\eps}{\partial y}\|^2_{L^2(Q)}\big),$$
and with Lemma \ref{2 transform} and Lemma \ref{3 transform} we get 
$$\leq C(\eta(\eps)\|f^\eps\|_{L^2(R^\eps)}^2+\eps\|f^\eps\|_{L^2(R^\eps)}^2)\leq C\eta(\eps)\|f^\eps\|_{L^2(R^\eps)}^2.$$

Putting all this inequalities together we prove the result. 


\section{Limit problem for the elliptic equation}

In view of Proposition \ref{main-reduction}, the homogenized limit problem of \eqref{1OPI0} will be obtained passing to the limit in the reduced problem \eqref{transformed-problem-hat}.  

We will obtain explicitly the homogenized limit problem of \eqref{transformed-problem-hat} for several interesting cases, in particular when the oscillating boundaries are given by
$$k^1_\eps(x)=h(x/\eps^\alpha), \quad k^2_\eps(x)=g(x/\eps^\beta),$$
where $0<\alpha, \beta<1$ and the functions $g,h \,: \R^n \to \R $ are $C^1$ periodic functions verifying 
\begin{align}\label{cota g+h}
&0\leq h_0\leq h(\cdot)\leq h_1,\\
&0<g_0 \leq g(\cdot)\leq g_1.
\end{align}

Note that these particular functions $k^1_\eps$ and $k^2_\eps$ satisfy hypothesis {\bf (H)} because of the fact that $0<\alpha,\beta<1$. Moreover, in this case, problem \ref{transformed-problem-hat} can be written as

 \begin{equation}\label{transformed-problem-hat2}
\left\{
\begin{gathered}
- \frac{1}{G_\eps}\sum_{i=1}^n\frac{\partial}{\partial x_i}\Big(G_\eps \frac{\partial \hat w^\eps}{\partial x_i}\Big)+ \hat w^\eps = \hat f^\eps
\quad \textrm{ in } \omega, \\
\frac{\partial \hat w^\eps}{\partial \eta}  = 0  \quad \textrm{ on } \partial \omega, 
\end{gathered}
\right. 
\end{equation}
where $G_\eps$ plays the role of $K_\eps$ and it is given by $G_\eps(x)= g(x/\eps^\beta) + h(x/\eps^\alpha)$.

From now on we will assume that $\hat f^\eps$ satisfies the following convergence
\begin{equation}\label{limit f1}
G_\eps(\cdot) \hat f^\eps(\cdot) =\frac{1}{\eps}\int_{-\eps \,h(\cdot/\eps^\alpha)}^{\epsilon \, g(\cdot/\eps^\beta)} f^\eps (\cdot, y) dy \weto f_0(\cdot) \quad \hbox{w-}L^2(\omega),
\end{equation}
 for certain $f_0\in L^2(\omega)$. 
%
%
%

We will study first the two-dimensional case adapting a classical simple argument in the context of periodic homogenization of the one-dimensional problems with oscillating coefficients, see \cite{BenLioPap,CiorDonato}. 
Secondly, we obtain the limit problem for the general and more complicated situation where the domain is $n-$dimensional, $n>2.$

\begin{remark}
Notice that considering this kind of periodic functions in problems with rough boundaries is very common, see \cite{C1,MP2, ArrVi2017} and the references therein. However, it is important to highlight that
we contemplate the possibility  of situations beyond the classical periodic setting in homogenization. For example, we consider cases where both boundaries oscillate with different rationally independent periods, which amounts to study a quasi-periodic problem. 
\end{remark}
\subsection{Two-dimensional case}
In this subsection we consider a two-dimensional thin domain $R^\epsilon$ which is given as the region between two oscillatory functions, that is, 
\begin{equation*}
R^\epsilon = \Big\{ (x, y) \in \R^2 \; | \;  x \in (0,1),  \; -\eps \,h(x/\eps^\alpha) < y < \epsilon \, g(x/\eps^\beta) \Big\},
 \end{equation*}
 where $0<\alpha, \beta<1$ and the functions $g,h \,: \R \to \R $ are $C^1$ periodic functions with period $L_1$ and $L_2$ respectively.
 Then, we study the behavior of the solutions of the Neumann problem \eqref{1OPI0}.

 In this case, problem \ref{transformed-problem-hat2} can be written as the following one dimensional problem 
\begin{equation}\label{transformed-problem21}
\left\{
\begin{gathered}
- \frac{1}{G_\eps} \frac{\partial}{\partial x}\Big(G_\eps \frac{\partial \hat w^\eps}{\partial x}\Big)  + \hat w^\eps = \hat f^\eps
\quad \textrm{ in } (0,1), \\
  (\hat w^\eps)'(0) =  (\hat w^\eps)'(1) =0, 
\end{gathered}
\right. 
\end{equation}
where $G_\eps(x)= g(x/\eps^\beta) + h(x/\eps^\alpha)$.
We would like to point that \eqref{transformed-problem21} presents the particularity of having not necessarily periodic coefficients. For instance, if $\alpha\ne \beta$ the problem is not periodic. Moreover, if $\alpha = \beta$ and the period of $g$ and $h$ are rationally independent periods then we have the situation of a quasi-periodic coefficients. 

The weak formulation of \eqref{transformed-problem21} is given by 
\begin{equation}\label{weak-version-final-problem1}
\int_{0}^1\Big\{ G_\eps \frac{\partial {\hat w}^\eps}{\partial x}\frac{\partial \phi^\eps}{\partial x} + G_\eps {\hat w}^\eps\phi \big\} \,  dx =\int_0^1 G_\eps \hat f^\eps \phi \, dx, \quad \hbox{ for all } \phi\in H^1(0,1).
\end{equation}

We start by establishing a priori estimates of ${\hat w}^\eps$. Considering ${\hat w}^\eps$ as a test function in \eqref{weak-version-final-problem1}, we easily get
$$||{\hat w}^\eps||_{H^1(0,1)}\leq C.$$
Thus, by weak compactness there exists $u_0 \in H^1(0,1)$ such that, up to subsequences
\begin{equation}\label{weak u1 appen}
 {\hat w}^\epsilon \weto u_0 \quad w-H^1(0, 1).
 \end{equation}

As in the simplest cases for the periodic homogenization, see for example \cite{BenLioPap, CiorDonato}, the key question now is: How is the limit of the product $G_\eps\frac{\partial {\hat w}^\eps}{\partial x}$? 
To solve this, we first obtain the weak limit of the functions $G_\eps$ and ${\displaystyle \frac{1}{G_\eps}.}$
 
 On one hand, since $G_\eps(x)=g(x/\eps^\beta) + h(x/\eps^\alpha)$ is the sum of two periodic functions it is obvious from the Average Convergence for Periodic Functions Theorem (see, e.g., \cite[p. xvi]{CioPau}) that $G_\eps(x)$ converges in a weak sense to the sum of the corresponding mean values, that is, 
 \begin{equation}\label{weak conv appen}
 G_\eps {\weto}  \frac{1}{L_1}\int_{0}^{L_1}g(y)\, dy + \frac{1}{L_2}\int_{0}^{L_2}h(z)\, dz \equiv {\cal M}(g)+{\cal M}(h) \quad w-L^2(0,1).
 \end{equation}
 
On the other hand, to prove the convergence ${\displaystyle \frac{1}{G_\eps}}$ we state the follwing lemma. 
\begin{lemma}
Let $G_\eps(x)=g(x/\eps^\beta) + h(x/\eps^\alpha)$. Then the following convergence holds 
$$\frac{1}{G_\eps}=\frac{1}{g(\frac{\cdot}{\eps^\beta})+h(\frac{\cdot}{\eps^\beta})}\weto  \frac{1}{p_0},$$
where $p_0$ is defined as follows
\begin{eqnarray*}
\frac{1}{p_0}
= 
 \left\{ 
 \renewcommand{\arraystretch}{2.5}
\begin{array}{ll}
{\displaystyle \lim_{T\to \infty} \frac{1}{T}\int_{0}^{T} \frac{1}{g(y) + h(y)}\, dy, \quad \hbox{if } \alpha=\beta,}  \\
 {\displaystyle\frac{1}{L_1L_2}\int_{0}^{L_1}\int_{0}^{L_2}\frac{1}{g(y) + h(z)}\, dz dy, \quad \hbox{if } \alpha \neq \beta.}
\end{array}
\right.
\end{eqnarray*}
\end{lemma}

\begin{proof}

We distinguish two different cases:

\par\noindent {i) Same order of oscillation ($\alpha = \beta$). }

 \begin{equation}\label{weak conv1 appen}
\frac{1}{g\Big(\frac{x}{\eps^\alpha}\Big)+ h\Big(\frac{x}{\eps^\alpha}\Big)} {\weto} \frac{1}{p_0} = \lim_{T \to \infty}\frac{1}{T}\int_{0}^{T}\frac{1}{g(y)+h(y)}\,dy  \quad w-L^2(0,1).
\end{equation}

We treat here the case where the function $G_\eps$ presents only one small scale, that is,
$$\frac{1}{G_\eps(x)}= \frac{1}{g(x/\eps^\alpha) + h(x/\eps^\alpha)}, \quad \hbox{for } x\in (0,1) \; \hbox{ and } \alpha \in (0,1).$$

Note that if the periods $L_1$ and $L_2$ are rationally dependent, there exist $p,q \in \N$ such that $pL_1 = q L_2$, then we immediately have from the  Average Convergence for Periodic Functions (see, e.g., \cite[p. xvi]{CioPau}) the weak convergence of $ \frac{1}{G_\eps}$
$$\frac{1}{G_\eps} {\weto} \frac{1}{p L_1} \int_0^{p L_1} \frac{1}{g(y) + h(y)}\,dy \quad w-L^2(0,1).$$

However, if $L_1$ and $L_2$ are rationally independent the usual periodicity hypothesis is replaced by a more general behavior: almost periodicity, see for example \cite{Be,Bohr}.
Indeed, in this case the function $\frac{1}{G(y)}=\frac{1}{g(y)+ h(y)}$ is not periodic 
but we show that it is almost periodic which allows us to obtain the weak limit. 

Since $G(y)= g(y) + h(y)$ is the sum of two periodic functions with different period we can ensure that $G$ is an almost periodic function. Then, from the definition of almost 
periodicity, for every $\eps>0$  there exists $T_0(\eps)$ such that every interval of length $T_0(\eps)$ contains a number $\tau$ with the following property:
$$|G(y+\tau) - G(y)|\leq m^2\eps, \; \hbox{ for each }y \in \R,$$
where $m$ is a constant such that $0<m\leq g(y) +h(y), \; \forall y \in \R$.

So we have,
\begin{align*}
\Big|\frac{1}{G(y+\tau)} - \frac{1}{G(y)}\Big|=\Big|\frac{G(y)-G(y+\tau) }{G(y+\tau)G(y)}\Big|\leq \frac{m^2 \eps}{m^2}=\eps,
\end{align*}
and hence  $\frac{1}{G(y)}$ is almost periodic.

Therefore, note that ${\displaystyle \lim_{T \to \infty}\frac{1}{T}\int_{0}^{T}\frac{1}{g(y)+h(y)}\,dy }$ is well defined since it is the mean value of the almost periodic function $\frac{1}{G(y)}$.

Now, we are in conditions to prove the desired weak convergence \eqref{weak conv1 appen}.

To obtain \eqref{weak conv1 appen}, since $||\frac{1}{G(y)}||_{L^\infty(0,1)}\leq \frac{1}{g_0 + h_0}$  and the set of all
the step functions is dense in $L^p(0,1), 1\leq p <\infty,$ it is enough to prove
 \begin{equation}\label{weak conv2 appen}
\lim_{\eps \to 0}\int_{a}^b\frac{1}{g\Big(\frac{x}{\eps^\alpha}\Big)+ h\Big(\frac{x}{\eps^\alpha}\Big)} \eto (b-a)  \lim_{T \to \infty}\frac{1}{T}\int_{0}^{T}\frac{1}{g(y)+h(y)}\,dy, \, \hbox{ for any } (a,b) \subset (0,1).
\end{equation}
We can write
\begin{equation}\label{decomp appen}
\int_{a}^b\frac{1}{g\Big(\frac{x}{\eps^\alpha}\Big)+ h\Big(\frac{x}{\eps^\alpha}\Big)}\,dx = \int_{0}^{b}\frac{1}{g\Big(\frac{x}{\eps^\alpha}\Big)+ h\Big(\frac{x}{\eps^\alpha}\Big)}\,dx - 
\int_{0}^{a}\frac{1}{g\Big(\frac{x}{\eps^\alpha}\Big)+ h\Big(\frac{x}{\eps^\alpha}\Big)}\,dx.
\end{equation}

By a simple change of variables we have
$$ \int_{0}^{e}\frac{1}{g\Big(\frac{x}{\eps^\alpha}\Big)+ h\Big(\frac{x}{\eps^\alpha}\Big)}\,dx = e\frac{\eps^\alpha}{e} \int_{0}^{e/\eps^\alpha}\frac{1}{g(y)+ h(y)}\,dy, \quad \forall e\in(0,1).$$
Then, since ${\displaystyle \frac{1}{g(y)+h(y)}\,dy }$ is almost periodic we can pass to the limit at the right-hand side of the last equality above to get
\begin{equation}\label{limite appen}
\lim_{\eps \to 0} \int_{0}^{e}\frac{1}{g\Big(\frac{x}{\eps^\alpha}\Big)+ h\Big(\frac{x}{\eps^\alpha}\Big)}\,dx  \eto e  \lim_{T \to \infty}\frac{1}{T}\int_{0}^{T}\frac{1}{g(y)+h(y)}\,dy,\quad \forall e\in(0,1).
\end{equation}
Finally, from \eqref{decomp appen} and \eqref{limite appen} we get convergence \eqref{weak conv2 appen}.

\par\noindent { ii) Different order of oscillation ($\alpha \neq \beta$). }

 \begin{equation}\label{weak conv2 appen1}
\frac{1}{g\Big(\frac{x}{\eps^\beta}\Big)+ h\Big(\frac{x}{\eps^\alpha}\Big)} {\weto} \frac{1}{p_0} =\frac{1}{L_1L_2}\int_{0}^{L_1}\int_{0}^{L_2}\frac{1}{g(y) + h(z)}\, dz dy \; w-L^2(0,1).
\end{equation}

Observe that in this case we are dealing with two microscopic scales which is a generalization of the classical result for periodic functions. The result is well-known in the literature,
see e.g. \cite{BenLioPap}.

\end{proof}

Now  we get the convergence of the product $G_\eps\frac{\partial {\hat w}^\eps}{\partial x}$.

Observe that  $G_\eps\frac{\partial {\hat w}^\eps}{\partial x}$ is uniformly bounded in $L^2(0,1)$ since 
$$\Big|\Big|\frac{\partial {\hat w}^\eps}{\partial x}\Big|\Big|_{L^2(0,1)}\leq C \quad  \hbox{ and } \quad 0<G_\eps(x)< g_1 +h_1, \hbox{ for each } x \in (0,1).$$
Moreover, taking into account that
$$\frac{\partial}{\partial x}\Big(G_\eps \frac{\partial {\hat w}^\eps}{\partial x}\Big) = -\hat f^\eps G_\eps+ G_\eps {\hat w}^\epsilon,
$$
we deduce that  $G_\eps\frac{\partial {\hat w}^\eps}{\partial x}$ is uniformly bounded in $H^1(0,1).$ Then, it follows that there exists a function $\sigma$ such that, up to subsequences,
$$G_\eps\frac{\partial {\hat w}^\eps}{\partial x} \longrightarrow \sigma \quad \hbox{strongly in } L^2(0,1).$$
Thus,
$$\frac{\partial {\hat w}^\eps}{\partial x} = \frac{1}{G_\eps} \Big(G_\eps\frac{\partial {\hat w}^\eps}{\partial x}\Big){\weto} \frac{1}{p_0} \sigma \quad w-L^2(0,1). $$

Consequently, due to convergence \eqref{weak u1 appen} we have
$$
\frac{\partial u_0}{\partial x}= \frac{1}{p_0} \sigma,$$
or equivalently,
\begin{equation}\label{product conv appen}
G_\eps\frac{\partial {\hat w}^\eps}{\partial x} \eto p_0 \frac{\partial u_0}{\partial x}\quad \hbox{strongly in } L^2(0,1).
\end{equation}

Therefore, in view of \eqref{weak u1 appen}, \eqref{limit f1} and \eqref{product conv appen} we can pass to the limit and we obtain the following weak formulation
$$\int_0^1 \Big\{p_0 \frac{\partial u_0}{\partial x}\frac{\partial \phi}{\partial x} + ({\cal M}(g)+{\cal M}(h)) u_0\phi\}\, dx = \int_0^1 f_0\phi\, dx.$$

Now we are in conditions to state the convergence result

\begin{theorem}\label{main appen}
Let $w^\eps$ be the solution of problem (\ref{1OPI0}). Then, with the definition of $f_0$ given by \eqref{limit f1} and denoting by $\hat f=\frac{f_0}{{\cal M}(g)+{\cal M}(h)}$,  then we have 
$$\hat w^\eps \to \hat w, \hbox{ w-}H^1(\omega),$$ 
$$\eps^{-1}||w^\eps - \hat w||_{L^2(R^\eps)} \to 0,$$
 where   $\hat w \in H^1(0,1)$ is the weak solution of  the following Neumann problem
\begin{equation} \label{homogenized problem1}
\left\{
\begin{gathered}
-\frac{p_0}{{\cal M}(g)+{\cal M}(h)}  {\hat w}_{xx} + \hat w = \hat f, \quad x \in (0,1), \\
\hat w'(0) = \hat w' (1) = 0,
\end{gathered}
\right.
\end{equation}
where the constant $p_0$ is such that 
\begin{equation}
 \frac{1}{h\Big(\frac{x}{\eps^\alpha}\Big) + g\Big(\frac{x}{\eps^\beta}\Big)} \weto \frac{1}{p_0} \quad w-L^2(0,1).
\end{equation}
 Therefore $p_0$ is given by

\begin{eqnarray*}
\frac{1}{p_0}
= 
 \left\{ 
 \renewcommand{\arraystretch}{2.5}
\begin{array}{ll}
{\displaystyle \lim_{T\to \infty} \frac{1}{T}\int_{0}^{T} \frac{1}{g(y) + h(y)}\, dy, \quad \hbox{if } \alpha=\beta,}  \\
 {\displaystyle\frac{1}{L_1L_2}\int_{0}^{L_1}\int_{0}^{L_2}\frac{1}{g(y) + h(z)}\, dz dy, \quad \hbox{if } \alpha \neq \beta.}
\end{array}
\right.
\end{eqnarray*}
\end{theorem}

\subsection{n-dimensional case}
Let $\omega \subset \R^n$ be a smooth domain. Then, we consider the following thin domain
\begin{equation*}
R^\epsilon = \Big\{ (x, y) \in \R^{n+1} \; | \;  x \in \omega,  \; -\eps \,h(x/\eps^\alpha) < y < \epsilon \, g(x/\eps^\beta) \Big\},
 \end{equation*}
 where $0<\alpha, \beta<1$ and the functions $g,h \,: \R^n \to \R $ are $C^1$ functions periodic in the cells $[0,L_1]^n$ and $[0, L_2]^n$ respectively.
We consider problem \eqref{transformed-problem-hat2} and 
%
taking ${\hat w}^\eps$ as test function in the variational formulation of \eqref{transformed-problem-hat2} we immediately get the a priori estimate
$$||{\hat w}^\eps||_{H^1(\omega)}\leq C.$$

Similarly to the two-dimensional case we consider two different situations:

\medskip
\par\noindent {\bf i) Same order of oscillation ($\alpha = \beta$) and periods rationally dependent}
Note that if both periods, $L_1$ and $L_2$, are rationally dependent, there exist $p,q \in \N$ such that $pL_1 = q L_2$, then \eqref{transformed-problem-hat2} is the classical problem in homogenization with periodic oscillating coefficients. Therefore, the following convergence result has been proved in the literature by using different techniques in homogenization, see e.g. Chapter 6 in \cite{CiorDonato} or \cite{BenLioPap,CioPau}.

\begin{proposition}
Let $w^\eps$ be the solution of problem (\ref{1OPI0}). Then, with the definition of $f_0$ given by \eqref{limit f1} and denoting by $\hat f=\frac{f_0}{{\cal M}(g)+{\cal M}(h)}$,  we have  
$$\hat w^\eps \to \hat w, \hbox{ w-}H^1(\omega),$$ 
$$\eps^{-1}||w^\eps - \hat w||_{L^2(R^\eps)} \to 0,$$
where $\hat w$ is the unique solution of the following Neumann problem
\begin{equation}\label{homogenized problem1}
\left\{
\begin{gathered}
 -\frac{1}{{\cal M}(g)+{\cal M}(h)} \hbox{div}(A_0  \nabla \hat w) + \hat w = \hat f
\quad \textrm{ in } \omega, \\
\frac{\partial \hat w}{\partial \eta}  = 0  \quad \textrm{ on } \partial \omega.
\end{gathered}
\right. 
\end{equation}

The matrix $A_0=(a_{ij})_{1\leq i,j \leq n}$ is constant and given by
\begin{align*}
a_{ii}= {\cal M}_{Y^*}\Big(G(z) \big(1- \frac{\partial X^i}{\partial z_i}\big)\Big),\\
a_{ij}= {\cal M}_{Y^*}\Big(-G(z) \frac{\partial X^j}{\partial z_i}\big)\Big),\\
\end{align*}
where $G(z)=g(z)+h(z)$, $Y^*=[0, pL_1]^n$ and $X^i$ is the unique solution of the following auxiliary problem
\begin{equation}\label{periodic auxiliary}
\left\{
\begin{gathered}
- \sum_{j=1}^n\frac{\partial}{\partial y_j}\Big(G \frac{\partial(X^i-z_i)}{\partial z_j}\Big)=0
\quad \textrm{ in } Y^*, \\
G\frac{\partial (X^i - z_i)}{\partial \eta}  = 0  \quad \textrm{ on } \partial Q,\\
{\cal M}_{Y^*}(X^i)=0,\\
X^i \quad  Y^*-\hbox{periodic}.
\end{gathered}
\right. 
\end{equation}
\end{proposition}

\par\noindent {\bf ii) Same order of oscillation ($\alpha = \beta$) and periods rationally independent}

However if the periods, $L_1$ and $L_2$, are rationally independent the coefficients of \eqref{transformed-problem-hat2} are rapidly oscillating quasi-periodic functions. Therefore, this specific case can be analyzed in the context of the more general almost periodic homogenization theory, see for instance \cite{OleiZi,Koz}. In fact, Proposition \ref{quasiperiodic} can be proved rigorously by applying directly the ideas introduced in \cite{Koz}. 

Note that, the main difference respect to periodic homogenization lies in the solvality of the auxiliary problem. 
To overcome this difficulty Kozlov lifted the auxiliary equation to a sub-elliptic problem on a higher dimensional torus which he solved thanks to a higher-order Poincar\'e inequality implied by the Diophantine condition. 

Therefore, since the auxiliary problem in our particular case can be degenerate,  the following ``frequency condition'' is necessary for the formation of the quasiperiodic solutions of the auxiliary problem

\medskip
\begin{itemize}
\item[{\bf Diophantine condition}] 
There exists $s_0 >0$ such that
$$|n_1L_1+n_2L_2| \geq \frac{C}{|n_1+n_2|^{s_0}}, \quad \forall (n_1,n_2) \in \N^2.$$
\end{itemize} 

\begin{proposition}\label{quasiperiodic}
Let $w^\eps$ be the solution of problem (\ref{1OPI0}). Then, with the definition of $f_0$ given by \eqref{limit f1} and denoting by $\hat f=\frac{f_0}{{\cal M}(g)+{\cal M}(h)}$,  we have 
$$\hat w^\eps \to \hat w, \hbox{ w-}H^1(\omega),$$ 
$$\eps^{-1}||w^\eps - \hat w||_{L^2(R^\eps)} \to 0,$$
where $\hat w$ is the unique solution of the following Neumann problem
\begin{equation}\label{homogenized problem1}
\left\{
\begin{gathered}
 -\frac{1}{{{\cal M}(g)+{\cal M}(h)}} \hbox{div}(A_0  \nabla \hat w) + \hat w = \hat f
\quad \textrm{ in } \omega, \\
\frac{\partial \hat w}{\partial \eta}  = 0  \quad \textrm{ on } \partial \omega.
\end{gathered}
\right. 
\end{equation}
The matrix $A_0=(a_{ij})_{1\leq i,j \leq n}$ is constant and it is given by
\begin{align*}
a_{ij}= \limsup_{L\to \infty}\frac{1}{(2L)^n}\int_{[-L,L]^n}\Big(-G(z) \frac{\partial X^j}{\partial z_i}\big)\Big)dz,\\
\end{align*}
where $G(z)=g(z)+h(z)$, $X^j$ and its derivatives are quasiperiodic functions which satisfy
\begin{equation}\label{periodic auxiliary22}
\left\{
\begin{gathered}
- \Big(\sum_{i=1}^n\frac{\partial}{\partial z_i}\Big(G \frac{\partial(X^j-z_j)}{\partial z_i}\Big)=0,\\
 \limsup_{L\to \infty}\frac{1}{(2L)^n}\int_{[-L,L]^n} X^j dz=0.
\end{gathered}
\right. 
\end{equation}
\end{proposition}

\par\noindent {\bf ii) Different order of oscillation ($\alpha \neq \beta$). }

In this case two microscopic scales appear, $\eps^\alpha$ and $\eps^\beta$ with $\alpha \neq \beta$. This means that each scale can be distinguished from the other, the frequency of the oscillations in both boundaries are not of the same order. Therefore, this problem must be studied in the framework of reiterated homogenization theory, see \cite{BenLioPap,AlBri}.

By using direclty the generalization of two-scale convergence given in \cite{AlBri} or the reiterated unfolding method \cite{MeuShaf05}   the homogenized problem for \eqref{transformed-problem3} is obtained. Then, we get the following homogenization result assuming without loss of generality that $\alpha$ is less than $\beta$.
\begin{proposition}
Let $w^\eps$ be the solution of problem (\ref{1OPI0}). Then, with the definition of $f_0$ given by \eqref{limit f1} and denoting by $\hat f=\frac{f_0}{{\cal M}(g)+{\cal M}(h)}$,  we have 
$$\hat w^\eps \to \hat w, \hbox{ w-}H^1(\omega),$$ 
$$\eps^{-1}||w^\eps - \hat w||_{L^2(R^\eps)} \to 0,$$
where $\hat w$ is the unique solution of the following Neumann problem
\begin{equation}\label{homogenized problem1}
\left\{
\begin{gathered}
 -\frac{1}{{{\cal M}(g)+{\cal M}(h)}} \hbox{div}(A_0  \nabla \hat w) +    \hat w = \hat f
\quad \textrm{ in } \omega, \\
\frac{\partial \hat w}{\partial \eta}  = 0  \quad \textrm{ on } \partial \omega,
\end{gathered}
\right. 
\end{equation}
where $A_0$ is a constant matrix defined by the inductive homogenization formula

\begin{itemize}
\item $A_2$ is a diagonal matrix of order $n$ with the function $G(x,z)=g(x)+h(z)$ in the elements of the diagonal.
\item $A_1$ is obtained by periodic homogenization of $A_2(x,\frac{x}{\eps^\beta})$.
\item $A_0$ is obtained by periodic homogenization of $A_1(\frac{x}{\eps^\alpha})$.
\end{itemize}
\end{proposition}

%
%

\section{Convergence properties of the semilinear parabolic equation}

In this section we show some convergence properties of the solutions and attractors of the evolutionary equations \eqref{1OPI0 parabolic} assuming that the functions which define the oscillating functions satisfy \eqref{per}.  In particular we analyze the relation between the semilinear parabolic problem defined in \eqref{1OPI0 parabolic} and its homogenized limit

\begin{equation}\label{homogenized problem par}
\left\{
\begin{gathered}
 u_{0_t}- \frac{1}{{\cal M}(g)+{\cal M}(h)} \hbox{div}(A_0  \nabla u_0) + u_0 = f(u_0)
\quad \textrm{ in } \omega, \\
\frac{\partial u_0}{\partial \eta}  = 0  \quad \textrm{ on } \partial \omega,
\end{gathered}
\right. 
\end{equation}
where $A_0$ depends on the dimension and the relation between $\alpha$ and $\beta$ as we have seen in the previous section.  The behavior of the nonlinear dynamics in thin domains with not necessarily oscillating boundaries  is not a new topic and we would like to refer to \cite{HR}, \cite{R}, \cite{ArrSan} for some works in this respect. Also, \cite{ArrCarPerSil} deals with the case of thin domains with oscillating boundaries. 

\begin{remark}
We provide only a sketch of the proofs of the results of this section because they are obtained by using standard arguments of a general approach discussed, for instance, in \cite{ACar,ACarLo1,ACarLo2,ACarLo3,ArrCarPerSil,CarPis,Per}.
\end{remark}  

The functional setting given by \cite{He81,Ha} allows to obtain the convergence of the linear semigroup given by \eqref{1OPI0 parabolic} to the one established by \eqref{homogenized problem par}. The concept of compact convergence that we adopt here was introduced in the works \cite{Stum1,Stum2,Stum3,Vai1,Vai2}.

First, we consider  a family of Hilbert spaces $\{Z_\eps\}_{\eps>0}$ defined by $Z_\eps=L^2(R^\eps)$ with the canonical inner product
$$(u,v)_\eps=\frac{1}{\eps}\int_{R^\eps} u(x,y)v(x,y)\, dx dy,$$
and let $Z_0=L^2(\omega)$ be the limit Hilbert  space with the following inner product
$$(u,v)_0=({\cal M}(g)+{\cal M}(h)) \int_{\omega}  u(x,y)v(x,y)\, dx.$$

Observe that the inner product in $Z_\eps$ has been scaled with a factor of $1/\eps$.  Hence, the induced norm will be also rescaled by the same factor. We will use the notation:
$$|||u|||_{Z_\eps}^2=\frac{1}{\eps}\|u\|_{L^2(R^\eps)}^2.$$

Now, we write the elliptic problem \eqref{1OPI0} as an abstract equation $L_\eps w^\eps=f^\eps$ where $L_\eps: \mathcal{D}(L_\eps)\subset Z_\eps \rightarrow  Z_\eps$ is the self adjoint, positive linear operator with compact resolvent defined as follows
$$\mathcal{D}(L_\eps)=\Big\{\omega^\eps \in H^2(R^\eps)| \frac{\partial w^\epsilon}{\partial \nu^\epsilon} = 0
\, \textrm{ on } \partial R^\epsilon \Big\},$$
\begin{equation}\label{opLe}
L_\eps w^\eps= - \Delta w^\epsilon + w^\epsilon, \quad w^\eps \in \mathcal{D}(L_\eps).
\end{equation}
Moreover, we denote by $Z_\eps^\alpha$ the fractional power scale associated to operators $L_\eps$, with $0\leq \alpha \leq 1$ and $0\leq \eps \leq 1$. Then, $Z_\eps=Z_\eps^0$ and 
$Z_\eps^{\frac{1}{2}}$ is the Sobolev Space $H^1(R^\eps)$ with norm $|||\cdot |||^2_{Z_\eps^{1/2}}=||\cdot ||^2_{H^1(R^\eps)}$.

Similarly, we rewrite the limit elliptic problem as $L_0 u_0=f$ where $L_0: \mathcal{D}(L_0)\subset Z_0 \rightarrow  Z_0$ is the self adjoint, positive linear operator with compact resolvent defined as follows
$$\mathcal{D}(L_0)=\Big\{u \in H^2(\omega)| \frac{\partial u}{\partial \nu} = 0
\, \textrm{ on } \partial \omega \Big\},$$
\begin{equation}\label{opL0}
L_0u= -\frac{1}{{\cal M}(g)+{\cal M}(h)} \hbox{div}(A_0  \nabla u) + u , \quad u \in \mathcal{D}(L_0),
\end{equation}
where $A_0$ is the matrix of homogenized coefficients obtained in the previous section. Notice that, $L_0$ is a positive self-adjoint operator with compact resolvent.

In the previous section we have passed to limit in the variational formulation of \eqref{VFP}. Now, we are in conditions to apply the concept of compact convergence to obtain convergence properties of the eigenvalues, eigenfunctions and the linear semigroups generated by $L_\eps$ and $L_0$.

Before recalling the main concepts of convergence we consider the family of linear continuous operators $E_\eps: Z_0 \rightarrow Z_\eps$ given by
$$E_\eps u(x,y)= u(x) \hbox{ on } \omega, \quad \forall u \in Z_0.$$
Then, it is clear $|||E_\eps u|||_{Z_\eps} \to \|u\|_{Z_0}.$ 
Analogously, we can consider $E_\eps: Z_0^1 \rightarrow Z_\eps^1$ and, taking in $Z_0^1$ the equivalent norm  $|| - \Delta u + u ||_{Z_0}$ we get
$$|||E_\eps u|||_{Z_\eps^1} \to \|u\|_{Z_0^1}.$$
Consequently, since 
$$\sup_{0\leq \eps \leq 1}\{\|E_\eps\|_{\mathcal{L}(Z_0,Z_\eps)}, \|E_\eps\|_{\mathcal{L}(Z_0^1,Z_\eps^1)}\} < \infty,$$
we get by interpolation that
$$C=\sup_{\eps>0}\|E_\eps\|_{\mathcal{L}(Z_0^\alpha,Z_\eps^\alpha)}<\infty \, \hbox{ for } 0\leq \alpha \leq 1. $$
Now we recall the main concept of convergence associated to the operators $\{E_\eps\}_{\eps>0}$.

\begin{definition}\label{ compact-Convergence}
We say that a family of compact operators $\{B_\eps \in \mathcal{L}(Z_\eps)| \eps>0\}$ converges compactly to a compact operator $B \in \mathcal{L}(Z_0)$, we write 
$B_\eps \stackrel{CC}{\longrightarrow}B$, if for any family $\{f^\eps\}_{\eps>0}$ with $|||f^\eps|||_{Z_\eps}\leq 1$ we have
\begin{itemize}
\item For each subsequence $\{B_{\eps_ m}f^{\eps_ m}\}$ of  a sequence $\{B_{\eps_ n}f^{\eps_ n}\}$,  $\eps_n \to 0$,  there exits a subsequence $\{B_{\eps_ {m'}}f^{\eps_ {m'}}\}$ and $F \in Z_0$ such that $|||B_{\eps_ {m'}}f^{\eps_ {m'}}-E_{\eps_{m'}}F |||_{Z_\eps} \to 0$.
\item There exists $B \in \mathcal{L}(Z_0)$ such that $|||B_{\eps} f^{\eps}-E_{\eps}Bf |||_{Z_\eps} \to 0$ if $||| f^{\eps}-E_{\eps}f |||_{Z_\eps} \to 0.$
\end{itemize} 
\end{definition}

Moreover, the following lemma holds, see Lemma 4.7 in \cite{ACarLo1}.
\begin{lemma}\label{key}
Assume that $\{B_\eps \in \mathcal{L}(Z_\eps)\}_{\eps \in (0,1]}$ converges compactly to $B$ as $\eps \to 0.$ Then, 
\begin{itemize}
\item[i)]$\|B_\eps\|_{\mathcal{L}(Z_\eps)}\leq C$, for some constant $C$ independent of $\eps$.
\item[ii)] Assume that $\mathcal{N}(I+B)=\{0\}$ then, there exists an $\eps_0>0$ and $M>0$ such that
$$\|(I+B_\eps)^{-1}\|_{\mathcal{L}(Z_\eps)}\leq M, \quad \forall \eps \in [0, \eps_0].$$
\end{itemize}
\end{lemma}

Then, the convergence results of previous section can be rewritten according to this framework. In particular, we have easily the following two results.

\begin{corollary} \label{comp}
The family of compact operators $\{L_\eps^{-1} \in \mathcal{L}(Z_\eps)\}_{\eps>0}$ converges compactly to the compact operator $L_0^{-1} \in \mathcal{L}(Z_0)$ as $\eps \to 0$. 
\end{corollary}

\begin{corollary}\label{resolvent}
Let $M_\eps: L^r(R^\eps) \rightarrow L^r(\omega)$ be the bounded linear operator given by 
$$ M_\eps f^\eps(x)= \frac{1}{\eps}\int_{-\eps \,h(x/\eps^\alpha)}^{\epsilon \, g(x/\eps^\beta)} f^\eps (x, y) dy, \quad x \in \omega.$$
Then, for each $\{f^\eps\} \subset Z_\eps$ with $|||f^\eps|||_{Z_\eps}$ uniformly bounded in $\eps$ there exists a subsequence such that the following convergence holds
$$|||L_\eps^{-1}f^\eps - E_\eps L_0^{-1}M_\eps f^\eps |||_{Z_\eps}\eto 0.  $$ 
Moreover, there exists $\eps_0>0$ and a function $\nu: (0, \eps_0)\rightarrow (0, \infty)$, with $\nu(\eps)\eto 0$, such that
\begin{equation}\label{estimatenu}
\|L_\eps^{-1} - E_\eps L_0^{-1}M_\eps \|_{\mathcal{L}(Z_\eps)}\leq \nu(\eps), \quad \forall \eps \in (0, \eps_0).
\end{equation}
\end{corollary}

Notice that Corollary \ref{comp} implies that
$L_\eps$ is a closed operator, has compact resolvent, zero belongs to its resolvent, denoted by $\rho(L_\eps)$, and $L_\eps^{-1} \stackrel{CC}{\longrightarrow}L_0^{-1}$. Moreover, Corollary \ref{resolvent} gives the convergence of the resolvent operators.

With this convergence we can show now the spectral convergence of the operators. Observe first that since the operators $L_\eps$, $L_=$  are selfadjoint and with compact resolvent, the spectrum is only discrete and real. Hence, let us denote by $\{\lambda_n^\eps\}_{n=1}^\infty\subset \R$ the set of eigenvalues ordered and counting multiplicity of the operator $L_\eps$ for $0\leq \eps\leq \eps_0$, for certain $\eps_0>0$ and let us denot by $\{\varphi_n^\eps\}_{n=1}^\infty$ a corresponding complete system of orthonormal eigenfunctions of $L_\eps$.    Notice also  that the corresponding set of eigenfucntions is not unique since we can always change the sign of an eigenfunction of a simple eigenvalue and in case an eigenvalue is multiple we have different choices to choose a base in the eigenspace associated to this multiple eigenvalue

Consequently, see Lemma 4.10 in \cite{ACarLo1} for the details, we have the following result about the spectrum convergence of operator $L_\eps$.
\begin{theorem}\label{spectral-convergence}
Let $L_\eps$ and $L_0$ the operators considered in \eqref{opLe} and \eqref{opL0} respectively. Then we have the spectral convergence of $L_\eps$ to $L_0$. 
That is, for each sequence $\eps_k\to 0$ there exists a subsequence, that we still denote by $\eps_k$, and a choice for a complete system of eigenfunctions of $L_0$, $\{\varphi_n^0\}_{n=1}^\infty$,   such that the following statements hold:
\begin{enumerate}
\item[i)] $\lambda_n^{\eps_k}\kti \lambda_n^0$ for each $n\in \N$, 
\item[ii)]  $|||\varphi_n^{\eps_k}-E_{\eps_k}\varphi_n^0|||_{Z_\eps^{1/2}}\kti 0$, for each $n\in \N$. 
\end{enumerate}
\end{theorem}

\begin{proof}  We will give an indication on how to obtain this result. The proof is based on the convergence of the spectral projections.  If $\lambda_0\in \R$ is an eigenvalue of $L_0$  of multiplicity $m\in \N$, (for instance $\lambda_{n}^0<\lambda_0=\lambda_{n+1}^0=\ldots=\lambda_{n+m}^0<\lambda_{n+m+1}^0$), then we have a small $\delta>0$ such that $\lambda_0$ is the unique spectral value in the set $\mathcal{O}(\lambda_0, \delta)=\{\lambda \in \mathcal(C): |\lambda-\lambda_0|\leq\delta\}$ and the projection over the eigenspace generated by $[\varphi_{n+1}^0,\ldots, \varphi_{n+m}^0]$ is given by 

$$Q_0(\lambda_0)=\frac{1}{2\pi i}\int_{|z-\lambda_0|=\delta}(zI-L_0)^{-1}dz,$$
which can be rewritten as
$$Q_0(\lambda_0)=L_0^{-1}\frac{1}{2\pi i}\int_{|z-\lambda_0|=\delta}(zL_0^{-1}-I)^{-1}dz.$$
Using now the compact convergence of $L_\eps^{-1}$ to $L_0^{-1}$ together with Lemma \ref{key}, allows us to show that the operators 
$$Q_\eps(\lambda_0)=\frac{1}{2\pi i}\int_{|z-\lambda_0|=\delta}(zI-L_\eps)^{-1}dz,$$
is well defined and  converges to the operator $Q_0$.  Both operators have finite dimensional range. This implies that necesarilly for $\eps$ small enough we have eigenvalues of $L_\eps$ in the set $\mathcal{O}(\lambda_0, \delta)$ and the combined multiplicity of all eigenvalues in this neighborhood is $m$. Moreover the convergence of the spectral projections imply the convergence of the eigenfunctions stated above.  We refer to \cite{ACarLo1, CarPis} for details. \end{proof}

We can also obtain the convergence of the linear semigroups generated by the operators $L_\eps$  and $L_0$, denoted by $e^{-L_\eps t}$ and $e^{-L_0 t}$.  


\begin{theorem}
There exists a function $\nu: (0, \eps_0] \rightarrow (0,\infty)$, $\nu(\eps) \eto 0$, and numbers $1/2<\gamma<1$, $0<b<1$ such that
\begin{equation}\label{estimatesemi}
||e^{-L_\eps t}- E_\eps e^{-L_0 t}M_\eps||_{\mathcal{L}(Z_\eps, Z_\eps^{1/2})}\leq \nu(\eps)e^{-bt}t^{-\gamma}, \quad \forall t>0.
\end{equation}
\end{theorem}
\begin{proof}
The proof of this result follows the same line of proof as Proposition 2.7 in \cite{ACar}. Using the spectral decomposition of the operators, we can write

$$e^{-L_\eps t}u_\eps=\sum_{n=1}^\infty e^{-\lambda_n^\eps t} (u_\eps, \varphi_n^\eps)_{L^2(R_\eps)}\varphi_n^\eps,$$ 
But the convergence of the eigenvalues and eigenfunctions obtained above in Theorem \ref{spectral-convergence} and with some computations as the ones performed in the proof of Proposition 2.7 in \cite{ACar} we obtain the result. 

Let us mention that the proof can also be done from a more functional analytic framework using the representation of the linear semigroups as 
$$e^{ L_\eps t}=\frac{1}{2\pi i}\int_{\tilde{\Gamma}} e^{\mu t} (\mu  I+ L_\eps )^{-1} d\mu,\quad 0\leq \eps\leq \eps_0,$$
where $\tilde \Gamma$ is the border of sector $\sum_{-1,\phi}=\{\mu \in \mathcal(C) : |arg (\mu+1)| \leq \phi\}$, $\frac{\pi}{2}<\phi<\pi$, oriented in such a way the imaginary part of $\mu$ increases as $\mu$ describes the curve $\tilde \Gamma$.  Using the convergence of the resolvent operators, Corollary \ref{comp} we can also prove the result, see \cite{ACarLo3} for instance. 
\end{proof}

%
%
%
%
%
%
%

Therefore, we have analyzed the behavior of the linear parts of problem \eqref{1OPI0 parabolic} as $\eps$ goes to zero. Finally, we show  the upper semicontinuity of the attractors and of the set of stationary states. This result is a consequence of the relation between the continuity of the linear semigroups with the continuity of the nonlinear semigroups through the Variation of Constants Formula.

Let $f:\R \rightarrow \R$ be a bounded $\mathcal{C}^2$-function with bounded derivatives up to second order and satisfying condition \eqref{dissipative}. Then, it is known that the solutions of problems \eqref{1OPI0 parabolic} and \eqref{homogenized problem par} are globally defined and we can associate to them the nonlinear semigroups $\{T_\eps(t)| 0\leq t\}$ and
 $\{T_0(t)| 0\leq t\}$.
\begin{theorem}
For each $\tau, R>0$ there exist a function $\bar \nu: (0, \eps_0] \rightarrow (0,\infty)$, $\bar\nu(\eps) \eto 0$, such that
\begin{equation}\label{convsemi}
|||T_\eps(t)w^\eps- E_\eps T_0(t)M_\eps w^\eps|||_{ Z_\eps^{1/2}}\leq \bar \nu(\eps) t^{-\gamma}, \quad \forall t \in (0,\tau), \|w^\eps\|_{Z^\eps}\leq R.
\end{equation}
Moreover, the family of attractors $\{ \mathcal{A}_\eps| \eps \in [0, \eps_0]\}$ of  problem \eqref{1OPI0 parabolic} is upper semicontinuous at $\eps=0$ in $Z_\eps^{1/2}$, in the sense that
\begin{equation}\label{convattractors}
\sup_{\varphi^\eps \in \mathcal{A}_\eps}\Big\{\inf_{\varphi \in \mathcal{A}_0}\{|||\varphi^\eps- E_\eps \varphi\}|||_{Z_\eps^{1/2}}\Big\}\eto 0.
\end{equation}
If $\mathcal{E}_\eps$ and $\mathcal{E}_0$ are the set of stationary states of problems \eqref{1OPI0 parabolic} and \eqref{homogenized problem par} respectively, then they satisfy the following convergence
\begin{equation}\label{convequilibria} 
 \sup_{\varphi^\eps \in \mathcal{E}_\eps}\Big\{\inf_{\varphi \in \mathcal{E}_0}\{|||\varphi^\eps- E_\eps \varphi\}|||_{Z_\eps^{1/2}}\Big\}\eto 0.
 \end{equation}
\end{theorem}
\begin{proof}
We follow the proof of Proposition 3.1 from \cite{ACar}.  To prove \eqref{convsemi} we use the Variations of Constant Formula, that is, 
$$T_\eps(t)w^\eps=e^{L_\eps t}w^\eps +\int_0^t e^{L_\eps(t-s)}f(T_\eps(s)w^\eps)ds, \quad 0\leq \eps\leq \eps_0.$$
Substractiong $T_\eps(t)w^\eps$ from $E_\eps T_0(t)M_\eps w_\eps$ and taking norms in $Z_\eps^{1/2}$ we obtain
\begin{align*}
|||T_\eps(t)w^\eps- E_\eps T_0(t)M_\eps w^\eps|||_{ Z_\eps^{1/2}}\leq |||e^{-L_\eps t}w^\eps- E_\eps e^{-L_0 t}M_\eps w^\eps|||_{Z_\eps^{1/2}}\\
+ \int_0^t |||e^{-L_\eps (t-s)} f(T_\eps(s)w^\eps)- E_\eps e^{-L_0 (t-s)} f(T_0(s)M_\eps w^\eps)|||_{Z_\eps^{1/2}}\, ds.
\end{align*}
Adding and substracting appropriate terms in the integral (as in Proposition 3.1 \cite{ACar})  using estimate \eqref{estimatesemi} and Gronwall inequality we get \eqref{convsemi}.

The upper semicontinuity of the attractors $ \mathcal{A}_\eps$ follows from 
the continuity of the nonlinear semigroups given by \eqref{convsemi}, the fact that 
$\cup_{0\leq \eps \leq \eps_0} M_\eps \mathcal{A}_\eps$ is a bounded set in $Z_0^{1/2}$,  that the attractor $\mathcal{A}_0$ attracts bounded sets (in particular it attracts $\cup_{0\leq \eps \leq \eps_0} M_\eps \mathcal{A}_\eps$) and the fact that  $\mathcal{A}_0$ is an invariant set for the flow $T_0(t)$, see for instance \cite{Ha, ACar}

we first note that $\cup_{0\leq \eps \leq \eps_0} M_\eps \mathcal{A}_\eps$ is a bounded set in $L^\infty(0,1)$ and in $Z_0$. Then, using the attractivity property of $\mathcal{A}_0$ in $Z_0$ we have that for any $\eta>0$ there exists $\tau>0$ such that 
$$ \inf_{\varphi \in \mathcal{A}_0}|||E_\eps T_0(\tau)M_\eps \varphi^\eps- E_\eps \varphi|||_{Z_\eps^\alpha} < \eta/2, \hbox{ for all } \varphi^\eps \in \mathcal{A}_\eps \hbox{ and }
0\leq \eps \leq \eps_0.$$
Consequently, \eqref{convattractors} is obtained using \eqref{convsemi} and taking into account $\mathcal{A}_\eps$ is an invariant set. 

Finally we show the upper semicontinuity of the set of stationary states $\mathcal{E}_\eps$. For this, we will show that for any sequence $\varphi^\eps \in \mathcal{E}_\eps$ we can get a subsequence that we still denote by $\varphi^\eps$ and a $\varphi^0\in \mathcal{E}_0$ such that 
$|||\varphi^\eps-E_\eps  \varphi^0|||_{Z_0^{1/2}}\to 0$.  But since $\mathcal{E}_\eps\subset \mathcal{A}_\eps$ and we have already proved the uppersemicontinuity of the attractors, we have the existence of $\varphi_0\in \mathcal{A}_0$ such that via a subsequence we have $|||\varphi^\eps-E_\eps  \varphi^0|||_{Z_0^{1/2}}\to 0$.  We need to prove that indeed $\varphi^0\in \mathcal{E}_0$.  But for this, observe that for any $t>0$ since $\varphi^\eps$ is a stationary state we have 
$$|||\varphi^\eps -E_\eps T_0(t)\varphi^0|||_{Z_\eps^{1/2}}=|||T_\eps(t)\varphi^\eps -E_\eps T_0(t)\varphi^0|||_{Z_\eps^{1/2}}\eto 0,$$
where we have used \eqref{convsemi}.  But since $\varphi^\eps- E_\eps \varphi^0\to 0$, we get  $\varphi^0 - T_0(t)\varphi^0=0$. This implies that $\varphi^0\in \mathcal{E}_0$, see Proposition 3.1, \cite{ACar}. 
\end{proof}
%

\end{document}